\documentclass[12pt,twoside]{amsart}
\usepackage{amssymb,amsmath,amsthm, amscd, enumerate, mathrsfs}
\usepackage{graphicx, hhline}
\usepackage[all]{xy}
\usepackage[usenames]{color}
\usepackage{hyperref}
\hypersetup{colorlinks=true}

\title{Pull-back of quasi-log structures}
\author{Osamu Fujino}
\date{2016/6/18, version 1.10}
\subjclass[2010]{Primary 14E30; Secondary 14J45.}
\keywords{quasi-log structures, log Fano pairs, fundamental groups}
\address{Department of Mathematics, Graduate School of Science,
Osaka University, Toyonaka, Osaka 560-0043, Japan}
\email{fujino@math.sci.osaka-u.ac.jp}

\newcommand{\Div}[0]{\operatorname{Div}}
\newcommand{\Pic}[0]{\operatorname{Pic}}
\newcommand{\Supp}[0]{\operatorname{Supp}}
\newcommand{\Spec}[0]{\operatorname{Spec}}
\newcommand{\Exc}[0]{\operatorname{Exc}}
\newtheorem{thm}{Theorem}[section]
\newtheorem{lem}[thm]{Lemma}
\newtheorem{cor}[thm]{Corollary}
\newtheorem{prop}[thm]{Proposition}
\newtheorem*{claim}{Claim}
\newtheorem{conj}[thm]{Conjecture}

\theoremstyle{definition}
\newtheorem{ex}[thm]{Example}
\newtheorem{defn}[thm]{Definition}
\newtheorem{rem}[thm]{Remark}
\newtheorem*{ack}{Acknowledgments}
\newtheorem{notation}[thm]{Notation}

\begin{document}
\bibliographystyle{amsalpha+}

\maketitle

\begin{abstract}
We prove that the pull-back of 
a quasi-log scheme by a smooth 
quasi-projective 
morphism 
has a natural quasi-log structure. 
We treat an application to log Fano pairs. 
This paper also contains a proof 
of the simple connectedness of 
log Fano pairs with only log canonical singularities 
by Kento Fujita. 
\end{abstract}

\section{Introduction}\label{f-sec1}

The following theorem is the main result 
of this paper, which is natural but 
missing in the literature. 
For the precise statement, see Theorem \ref{f-thm3.5} below. 

\begin{thm}[Pull-back of quasi-log structures]\label{f-thm1.1}  
Let $[X, \omega]$ be a quasi-log scheme and let 
$h:X'\to X$ be a smooth quasi-projective 
morphism. 
Then $[X', \omega']$, where $\omega'=h^*\omega\otimes \omega_{X'/X}$ 
with $\omega_{X'/X}=\det \Omega^1_{X'/X}$, 
has a natural quasi-log structure induced by $h$. 

In particular, if $h$ is 
a finite \'etale morphism, 
then $[X', \omega']$, where $\omega'=h^*\omega$, 
has a natural quasi-log structure induced by $h$. 
\end{thm}

We make an important remark:~we 
do not know whether Theorem \ref{f-thm1.1} holds true or not 
without assuming that $h$ is {\em{quasi-projective}}. 
As an easy application of Theorem \ref{f-thm1.1}, we obtain: 

\begin{cor}\label{f-cor1.2}
Let $[X, \omega]$ be a projective quasi-log canonical pair such that 
$-\omega$ is ample. 
Then the algebraic fundamental group of $X$ is trivial, or 
equivalently, 
$X$ has no nontrivial finite \'etale covers. 
\end{cor}

In view of Corollary \ref{f-cor1.2}, it is natural to conjecture: 

\begin{conj}\label{f-conj1.3}
Let $[X, \omega]$ be a projective 
quasi-log canonical pair such that $-\omega$ is ample. 
Then $X$ is simply connected. 
\end{conj}

In general, there exists an irreducible projective 
variety whose algebraic fundamental group is trivial 
and whose topological fundamental group is nontrivial 
(Example \ref{f-ex5.4}). 
As a special case of Conjecture \ref{f-conj1.3}, we have: 

\begin{conj}\label{f-conj1.4} 
Let $(X, \Delta)$ be a projective semi-log canonical pair 
such that $-(K_X+\Delta)$ is ample. 
Then $X$ is simply connected. 
\end{conj}

It is well known that Conjecture \ref{f-conj1.4} 
holds when $(X, \Delta)$ is kawamata log terminal (\cite{takayama}). 
Kento Fujita pointed out that Conjecture \ref{f-conj1.4} 
holds true when $(X, \Delta)$ is log canonical. 

\begin{thm}[Fujita, Theorem \ref{f-thm6.1}]\label{f-thm1.5} 
Let $(X, \Delta)$ be a projective log canonical 
pair such that $-(K_X+\Delta)$ is ample. 
Then $X$ is simply connected.  
\end{thm}

\begin{ack}
The author would like to thank Takeshi 
Abe, Kento Fujita, Yuichiro Hoshi, and Tetsushi Ito 
for answering his questions and giving him useful comments. 
He was partially supported by Grant-in-Aid 
for Young Scientists (A) 24684002, 
Grant-in-Aid for Scientific Research (S) 
16H06337, and (B) 16H03925 
from JSPS. 
He thanks the referees for useful comments.  
\end{ack} 

We work over $\mathbb C$, the complex number field, throughout this 
paper. We recommend the reader to see \cite{fujino-intro} 
for a gentle introduction to the theory of quasi-log structures. 
We will never publish \cite{fujino-book}. 
Therefore, we reproduce some of the arguments from \cite{fujino-book} 
in the current paper. For the basic definitions and properties of 
semi-log canonial pairs, see \cite{fujino-slc}. 

\section{Preliminaries}\label{f-sec2} 

\begin{notation}\label{f-notation2.1}
A pair $[X, \omega]$ consists of a scheme $X$ and an $\mathbb R$-Cartier 
divisor (or $\mathbb R$-line bundle) $\omega$ on $X$. 
In this paper, a scheme means a separated scheme of finite type over 
$\Spec \mathbb C$. A variety is a reduced scheme. 
\end{notation}

\begin{notation}[Divisors]\label{f-notation2.2} 
Let $B_1$ and $B_2$ be two $\mathbb R$-Cartier divisors 
on a scheme $X$. Then $B_1$ 
is linearly (resp.~$\mathbb Q$-linearly, or $\mathbb R$-linearly) equivalent to 
$B_2$, denoted by 
$B_1\sim B_2$ (resp.~$B_1\sim _{\mathbb Q} B_2$, 
or $B_1\sim _{\mathbb R}B_2$) if 
$
B_1=B_2+\sum _{i=1}^k r_i (f_i)
$ 
such that $f_i \in \Gamma (X, \mathcal K_X^*)$ 
and $r_i\in \mathbb Z$ (resp.~$r_i \in \mathbb Q$, or $r_i \in 
\mathbb R$) for every $i$. Here, $\mathcal K_X$ 
is the sheaf of total quotient rings of 
$\mathcal O_X$ and $\mathcal K_X^*$ is 
the sheaf of invertible elements in the sheaf of rings $\mathcal K_X$. 
We note that 
$(f_i)$ is a {\em{principal Cartier divisor}} associated to $f_i$, that is, 
the image of $f_i$ by 
$
\Gamma (X, \mathcal K_X^*)\to\Gamma (X, \mathcal K_X^*/\mathcal O_X^*)$, 
where $\mathcal O_X^*$ is the sheaf of invertible elements in $\mathcal O_X$. 

Let $D$ be a $\mathbb Q$-divisor (resp.~an $\mathbb R$-divisor) 
on an equidimensional variety $X$, that is, 
$D$ is a finite formal $\mathbb Q$-linear (resp.~$\mathbb R$-linear) 
combination 
$D=\sum _i d_i D_i$ of irreducible 
reduced subschemes $D_i$ of codimension one. 
We define the {\em{round-up}} 
$\lceil D\rceil =\sum _i \lceil d_i \rceil D_i$ (resp.~{\em{round-down}} 
$\lfloor D\rfloor =\sum _i \lfloor d_i \rfloor D_i$), where 
every real number $x$, $\lceil x\rceil$ (resp.~$\lfloor x\rfloor$) is the integer 
defined by $x\leq \lceil x\rceil <x+1$ 
(resp.~$x-1<\lfloor x\rfloor \leq x$). The 
{\em{fractional part}} $\{D\}$ of $D$ denotes $D-\lfloor D\rfloor$. We put 
$D^{<1}=\sum _{d_i<1}d_i D_i$, 
$D^{\leq 1}=\sum _{d_i\leq1}d_i D_i$, and $D^{=1}=\sum _{d_i=1}D_i$. 
We can define $D^{\geq 1}$, $D^{>1}$, 
and so on, analogously. 
We call $D$ a {\em{boundary}} (resp.~{\em{subboundary}}) $\mathbb R$-divisor if 
$0\leq d_i\leq 1$ (resp.~$d_i\leq 1$) for every $i$. 
\end{notation}
\begin{notation}[Singularities of pairs]\label{f-notation2.3}
Let $X$ be a normal variety and let $\Delta$ be an 
$\mathbb R$-divisor on $X$ 
such that $K_X+\Delta$ is $\mathbb R$-Cartier. 
Let $f:Y\to X$ be 
a resolution such that $\Exc(f)\cup f^{-1}_*\Delta$, 
where $\Exc (f)$ is the exceptional locus of $f$ 
and $f^{-1}_*\Delta$ is 
the strict transform of $\Delta$ on $Y$,  
has a simple normal crossing support. We can 
write 
$K_Y=f^*(K_X+\Delta)+\sum _i a_i E_i$. 
We say that $(X, \Delta)$ 
is {\em{sub log canonical}} ({\em{sub lc}}, for short) if $a_i\geq -1$ for every $i$. 
We usually write $a_i= a(E_i, X, \Delta)$
and call it the {\em{discrepancy coefficient}} of 
$E_i$ with respect to $(X, \Delta)$. 
It is well known that there exists the largest Zariski open set $U$ of $X$ such that 
$(U, \Delta|_U)$ is sub log canonical. 
If there exist a resolution $f:Y\to X$ and a divisor $E$ on $Y$ such 
that $a(E, X, \Delta)=-1$ and $f(E)\cap U\ne \emptyset$, then $f(E)$ is called a 
{\em{log canonical center}} (an {\em{lc center}}, 
for short) with respect to $(X, \Delta)$. 
If $(X, \Delta)$ is sub log canonical and $\Delta$ is effective, then 
$(X, \Delta)$ is called {\em{log canonical}} ({\em{lc}}, for short). 

We note that we can define $a(E_i, X, \Delta)$ in more general 
settings (\cite[Definition 2.4]{kollar2}).  
\end{notation}

Let us recall the definition of simple normal crossing pairs. 

\begin{defn}[Simple normal crossing pairs]\label{f-def2.4} 
We say that the pair $(X, D)$ is {\em{simple normal crossing}} at 
a point $a\in X$ if $X$ has a Zariski open 
neighborhood $U$ of $a$ that can be embedded in a smooth 
variety 
$Y$, 
where $Y$ has regular system of 
parameters $(x_1, \cdots, x_p, y_1, \cdots, y_r)$ at 
$a=0$ in which $U$ is defined by a monomial equation 
$
x_1\cdots x_p=0
$ 
and $
D=\sum _{i=1}^r \alpha_i(y_i=0)|_U$ 
with $\alpha_i\in \mathbb R$. 
We say that $(X, D)$ is a {\em{simple 
normal crossing pair}} if it is simple normal crossing at every point of $X$. 
We say that a simple normal crossing pair $(X, D)$ is {\em{embedded}} 
if there exists a closed embedding $\iota:X\to M$, where 
$M$ is a smooth variety of $\dim X+1$. 
We call $M$ the {\em{ambient space}} of $(X, D)$. 
If $(X, 0)$ is a simple normal crossing 
pair, then $X$ is called a {\em{simple normal crossing 
variety}}. If $X$ is a simple normal 
crossing variety, then $X$ has only Gorenstein singularities. 
Thus, it has an invertible dualizing sheaf $\omega_X$. 
Therefore, we can define the {\em{canonical divisor $K_X$}} such that 
$\omega_X\simeq \mathcal O_X(K_X)$. 
It is a Cartier divisor on $X$ and is well-defined up to linear equivalence. 

Let $X$ be a simple normal crossing variety and let $X=\bigcup _{i\in I}X_i$ be the 
irreducible decomposition of $X$. 
A {\em{stratum}} of $X$ is an irreducible component of $X_{i_1}\cap \cdots \cap X_{i_k}$ for some 
$\{i_1, \cdots, i_k\}\subset I$. 

Let $X$ be a simple normal crossing variety and 
let $D$ be a Cartier divisor on $X$. 
If $(X, D)$ is a simple normal crossing pair and $D$ is reduced, 
then $D$ is called a {\em{simple normal crossing divisor}} on $X$. 

Let $(X, D)$ be a simple normal crossing pair. 
Let $\nu:X^\nu \to X$ be the normalization. 
We define $\Theta$ by the formula 
$
K_{X^\nu}+\Theta=\nu^*(K_X+D)$.  
Then a {\em{stratum}} of $(X, D)$ is 
an irreducible component of $X$ or the $\nu$-image 
of a log canonical center of $(X^\nu, \Theta)$ 
(Notation \ref{f-notation2.3}). 
When $D=0$, 
this definition is compatible with the above definition of the strata of $X$. 
When $D$ is a boundary $\mathbb R$-divisor, 
$W$ is a stratum of $(X, D)$ if and only if 
$W$ is an slc stratum of $(X, D)$ 
(\cite[Definition 2.5]{fujino-slc}). Note that 
$(X, D)$ is semi-log canonical if $D$ is a boundary $\mathbb R$-divisor. 
\end{defn}

\begin{notation}
$\pi_1(X)$ denotes the topological fundamental group of 
$X$. 
\end{notation}

\section{Pull-back of quasi-log structures}\label{f-sec3} 

In this section, we give a precise statement of Theorem \ref{f-thm1.1} 
(Theorem \ref{f-thm3.5}). 
First, let us recall the definition of {\em{globally embedded 
simple normal crossing pairs}} in order to define quasi-log schemes. 

\begin{defn}[Globally embedded simple normal crossing 
pairs]\label{def3.1} 
Let $Y$ be a simple normal crossing divisor 
on a smooth 
variety $M$ and let $D$ be an $\mathbb R$-divisor 
on $M$ such that 
$\Supp (D+Y)$ is a simple normal crossing divisor on $M$ and that 
$D$ and $Y$ have no common irreducible components. 
We put $B_Y=D|_Y$ and consider the pair $(Y, B_Y)$. 
We call $(Y, B_Y)$ a {\em{globally embedded simple normal 
crossing pair}} and $M$ the {\em{ambient space}} of $(Y, B_Y)$. 
\end{defn}

It is obvious that a globally embedded simple normal crossing 
pair is an embedded simple normal crossing 
pair in Definition \ref{f-def2.4}. 

Let us define {\em{quasi-log schemes}}. 
For Ambro's original definition in \cite{ambro}, 
see Definition \ref{f-def7.2} 
below. 

\begin{defn}[Quasi-log schemes]\label{def3.2}
A {\em{quasi-log scheme}} is a scheme $X$ endowed with an 
$\mathbb R$-Cartier divisor 
(or $\mathbb R$-line bundle) 
$\omega$ on $X$, a proper closed subscheme 
$X_{-\infty}\subset X$, and a finite collection $\{C\}$ of reduced 
and irreducible subschemes of $X$ such that there is a 
proper morphism $f:(Y, B_Y)\to X$ from a globally 
embedded simple 
normal crossing pair satisfying the following properties: 
\begin{itemize}
\item[(1)] $f^*\omega\sim_{\mathbb R}K_Y+B_Y$. 
\item[(2)] The natural map 
$\mathcal O_X
\to f_*\mathcal O_Y(\lceil -(B_Y^{<1})\rceil)$ 
induces an isomorphism 
$$
\mathcal I_{X_{-\infty}}\overset{\simeq}{\longrightarrow} f_*\mathcal O_Y(\lceil 
-(B_Y^{<1})\rceil-\lfloor B_Y^{>1}\rfloor),  
$$ 
where $\mathcal I_{X_{-\infty}}$ is the defining ideal sheaf of 
$X_{-\infty}$. 
\item[(3)] The collection of subvarieties $\{C\}$ coincides with the image 
of $(Y, B_Y)$-strata that are not included in $X_{-\infty}$. 
\end{itemize}
We simply write $[X, \omega]$ to denote 
the above data 
$
\bigl(X, \omega, f:(Y, B_Y)\to X\bigr)
$ 
if there is no risk of confusion. 
Note that a quasi-log scheme $X$ is the union of $\{C\}$ and $X_{-\infty}$. 
We also note that $\omega$ is called the {\em{quasi-log canonical class}} 
of $[X, \omega]$, which is defined up to $\mathbb R$-linear equivalence.  
We sometimes simply say that 
$[X, \omega]$ is a {\em{quasi-log pair}}. 
The subvarieties $C$ 
are called the {\em{qlc strata}} of $[X, \omega]$, 
$X_{-\infty}$ is called the {\em{non-qlc locus}} 
of $[X, \omega]$, and $f:(Y, B_Y)\to X$ is 
called a {\em{quasi-log resolution}} 
of $[X, \omega]$. 
\end{defn}

\begin{rem}\label{f-rem3.3} 
Let $\Div(Y)$ be the group of Cartier divisors 
on $Y$ and let $\Pic (Y)$ be the Picard group of $Y$. 
Let $
\delta_Y:\Div(Y)\otimes \mathbb R\to \Pic(Y)\otimes\mathbb R
$ 
be the homomorphism 
induced by $A\mapsto \mathcal O_Y(A)$ 
where $A$ is a Cartier divisor on $Y$. 
When $\omega$ is an $\mathbb R$-line bundle in Definition \ref{def3.2}, 
$
f^*\omega\sim _{\mathbb R}K_Y+B_Y
$ 
means 
$
f^*\omega=\delta_Y(K_Y+B_Y)
$ 
in $\Pic(Y)\otimes \mathbb R$. 
Even when $\omega$ is an $\mathbb R$-line bundle, 
we use $-\omega$ to denote the inverse of $\omega$ in 
$\Pic (X)\otimes \mathbb R$ 
if there is no risk of confusion. 
If $\omega$ is an $\mathbb R$-Cartier divisor 
on $X$ in Theorem \ref{f-thm1.1}, $h^*\omega\otimes 
\det \Omega_{X'/X}^1$ means 
$
\delta_{X'}(h^*\omega)\otimes \det \Omega_{X'/X}^1
$ 
in $\Pic (X')\otimes \mathbb R$ where $\delta_{X'}:\Div (X')\otimes 
\mathbb R\to \Pic(X')\otimes \mathbb R$. 
\end{rem}

For various applications, the notion of {\em{qlc pairs}} 
is very useful. 

\begin{defn}\label{f-def3.4}
Let $[X, \omega]$ be a quasi-log pair. 
We say that $[X, \omega]$ has only {\em{quasi-log canonical singularities}} 
({\em{qlc singularities}}, for short) or $[X, \omega]$ is a {\em{qlc pair}} 
if $X_{-\infty}=\emptyset$. 
\end{defn}

Let us state the main theorem of this paper 
precisely. 

\begin{thm}[Main theorem]\label{f-thm3.5}
Let $[X, \omega]$ be a quasi-log pair as in Definition \ref{def3.2}. 
Let $X'$ be a scheme and let $h:X'\to X$ be a 
smooth quasi-projective morphism. 
Then $[X', \omega']$, where $\omega'=h^*\omega\otimes \omega_{X'/X}$ with 
$\omega_{X'/X}=\det \Omega^1_{X'/X}$, has a natural quasi-log 
structure induced by $h$. More precisely, we have the following: 
\begin{itemize}
\item[(i)] {\em{(Non-qlc locus)}}.~There 
is a proper closed subscheme $X'_{-\infty}\subset X'$.  
\item[(ii)] {\em{(Quasi-log resolution)}}.~There exists a proper morphism 
$f':(Y', B_{Y'})\to X'$ from 
a globally embedded simple normal crossing pair $(Y', B_{Y'})$ 
with $f'^*\omega'\sim _{\mathbb R} K_{Y'}+B_{Y'}$ which defines a quasi-log structure on $[X', \omega']$ 
such that 
$\mathcal I_{X'_{-\infty}}=h^*\mathcal I_{X_{-\infty}}$. 
\item[(iii)] {\em{(Qlc strata)}}.~There is a finite collection $\{C'\}$ of reduced 
and irreducible subschemes of $X'$ such that $\{C'\}=\{f^{-1}(C)\}$ 
and that the collection of subvarieties $\{C'\}$ coincides with 
the images of $(Y', B_{Y'})$-strata that are not included in $X'_{-\infty}$. 
\end{itemize}
\end{thm}

For the definition and basic properties of {\em{quasi-projective 
morphisms}}, see \cite[Chapitre II \S 5.3.~Morphismes 
quasi-projectifs]{ega}. 

\section{On quasi-log structures}\label{sec4}

\begin{prop}[{\cite[Proposition 3.50]{fujino-book}}]\label{f-prop4.1}
Let $f:V\to W$ be a proper birational morphism between 
smooth varieties and let $B_W$ be an 
$\mathbb  R$-divisor on $W$ such 
that $\Supp B_W$ is a simple normal crossing divisor on $W$. 
Assume that $K_V+B_V=f^*(K_W+B_W)$ and 
that $\Supp B_V$ is a simple normal crossing divisor on $V$. 
Then we have $$f_*\mathcal O_V(\lceil -(B^{<1}_V)\rceil 
-\lfloor B^{>1}_V\rfloor)\simeq 
\mathcal O_W(\lceil -(B^{<1}_W)\rceil 
-\lfloor B^{>1}_W\rfloor). $$ 
Furthermore, 
let $S$ be a simple normal crossing divisor on $W$ such 
that $S\subset \Supp B^{=1}_W$. Let $T$ be the union of the 
irreducible 
components of $B^{=1}_V$ that are mapped into $S$ by $f$. 
Assume that 
$\Supp f^{-1}_*B_W\cup \Exc (f)$ is a simple normal crossing 
divisor on $V$. 
Then we have 
$$f_*\mathcal O_T(\lceil -(B^{<1}_T)\rceil 
-\lfloor B^{>1}_T\rfloor)\simeq 
\mathcal O_S(\lceil -(B^{<1}_S)\rceil 
-\lfloor B^{>1}_S\rfloor),$$ 
where 
$(K_V+B_V)|_T=K_T+B_T$ and $(K_W+B_W)|_S=K_S+B_S$. 
\end{prop}
\begin{proof}
By $K_V+B_V=f^*(K_W+B_W)$, we 
obtain 
$$
K_V=f^*(K_W+B^{=1}_W+\{B_W\})+f^*
(\lfloor B^{<1}_W\rfloor+\lfloor B^{>1}_W\rfloor)
-(\lfloor B^{<1}_V\rfloor+\lfloor B^{>1}_V\rfloor)
-B^{=1}_V-\{B_V\}.
$$
If $a(\nu, W, B^{=1}_W+\{B_W\})=-1$ for a prime divisor 
$\nu$ over $W$, then 
we can check that $a(\nu, W, B_W)=-1$ by using 
\cite[Lemma 2.45]{kollar-mori}. 
Since $f^*
(\lfloor B^{<1}_W\rfloor+\lfloor B^{>1}_W\rfloor)
-(\lfloor B^{<1}_V\rfloor+\lfloor B^{>1}_V\rfloor)$ is 
Cartier, we can easily see that 
$f^*(\lfloor B^{<1}_W\rfloor+\lfloor B^{>1}_W\rfloor)
=\lfloor B^{<1}_V\rfloor+\lfloor B^{>1}_V\rfloor+E$,  
where $E$ is an effective $f$-exceptional divisor. 
Thus, we obtain 
$f_*\mathcal O_V(\lceil -(B^{<1}_V)\rceil 
-\lfloor B^{>1}_V\rfloor)\simeq 
\mathcal O_W(\lceil -(B^{<1}_W)\rceil 
-\lfloor B^{>1}_W\rfloor)$.   
Next, we consider the short exact sequence: 
$$
0\to \mathcal O_V(\lceil -(B^{<1}_V)\rceil-
\lfloor B^{>1}_V\rfloor-T)\to 
\mathcal O_V(\lceil -(B^{<1}_V)\rceil-
\lfloor B^{>1}_V\rfloor)
\to \mathcal O_T(\lceil -(B^{<1}_T)\rceil-
\lfloor B^{>1}_T\rfloor)\to 0. 
$$
Since $T=f^*S-F$, where $F$ is an effective $f$-exceptional 
divisor, we can easily see that 
$$f_*\mathcal O_V(\lceil -(B^{<1}_V)\rceil 
-\lfloor B^{>1}_V\rfloor-T)\simeq 
\mathcal O_W(\lceil -(B^{<1}_W)\rceil 
-\lfloor B^{>1}_W\rfloor-S).$$  
We note that 
$
(\lceil -(B^{<1}_V)\rceil 
-\lfloor B^{>1}_V\rfloor-T)-(K_V+\{B_V\}+B^{=1}_V-T)
=-f^*(K_W+B_W)$. 
Therefore, every associated prime 
of $R^1f_*\mathcal O_V(\lceil -(B^{<1}_V)\rceil 
-\lfloor B^{>1}_V\rfloor-T)$ is the generic point of the $f$-image 
of some stratum of $(V, \{B_V\}+B^{=1}_V-T)$ by 
\cite[Theorem 6.3 (i)]{fujino}. 
\begin{claim}\label{f-claim}
No strata of $(V, \{B_V\}+B^{=1}_V-T)$ are 
mapped into $S$ by $f$. 
\end{claim}
\begin{proof}[Proof of Claim]
Assume that there is a stratum $C$ of $(V, \{B_V\}+B^{=1}_V-T)$ such that 
$f(C)\subset S$. Note that 
$\Supp f^*S\subset \Supp f^{-1}_*B_W\cup \Exc (f)$ and 
$\Supp B^{=1}_V\subset \Supp f^{-1}_*B_W\cup \Exc (f)$. 
Since $C$ is also a stratum of $(V, B^{=1}_V)$ and 
$C\subset \Supp f^*S$,  
there exists an irreducible component $G$ of $B^{=1}_V$ such that 
$C\subset G\subset \Supp f^*S$.  
Therefore, by the definition of $T$, $G$ is an 
irreducible component of $T$ because $f(G)\subset S$ and $G$ is an 
irreducible component of $B^{=1}_V$. So, 
$C$ is not a stratum of $(V, \{B_V\}+B^{=1}_V-T)$. It is 
a contradiction. 
\end{proof}
On the other hand, $f(T)\subset S$. Therefore, the connecting 
homomorphism  
$$f_*\mathcal O_T(\lceil -(B^{<1}_T)\rceil
-\lfloor B^{>1}_T\rfloor)
\to R^1f_*\mathcal O_V(\lceil -(B^{<1}_Z)\rceil
-\lfloor B^{>1}_Z\rfloor-T)$$ is a zero map 
by Claim. Thus, we obtain 
$f_*\mathcal O_T(\lceil -(B^{<1}_T)\rceil 
-\lfloor B^{>1}_T\rfloor)\simeq 
\mathcal O_S(\lceil -(B^{<1}_S)\rceil 
-\lfloor B^{>1}_S\rfloor)$ 
by an easy diagram chasing. We finish the proof.  
\end{proof}

It is easy to check: 

\begin{prop}\label{f-prop4.2}
In Proposition \ref{f-prop4.1}, 
let $C'$ be a log canonical center of $(V, B_V)$ contained in $T$. 
Then $f(C')$ is a log canonical 
center of $(W, B_W)$ contained in $S$ or $f(C')$ is contained 
in $\Supp B_W^{>1}$. 
Let $C$ be a log canonical center of $(W, B_W)$ contained in $S$. Then 
there exists a log canonical center $C'$ of $(V, B_V)$ contained in $T$ such that 
$f(C')=C$. 
\end{prop}

\begin{thm}\label{f-thm4.3}
In Definition \ref{def3.2}, we may assume that 
the ambient space $M$ of 
the globally embedded simple normal crossing 
pair $(Y, B_Y)$ is quasi-projective. 
In particular, $Y$ is quasi-projective. 
\end{thm}
\begin{proof}
In Definition \ref{def3.2}, we may assume that 
$D+Y$ is an $\mathbb R$-divisor 
on a smooth variety $M$ such that 
$\Supp (D+Y)$ is a simple normal crossing divisor 
on $M$, $D$ and $Y$ have no common irreducible components, 
and $B_Y=D|_Y$ as in Definition \ref{def3.1}. 
Let $g:M'\to M$ be a projective 
birational morphism from a smooth quasi-projective 
variety $M'$ with the following properties: 
\begin{itemize}
\item[(i)] $K_{M'}+B_{M'}=g^*(K_M+D+Y)$, 
\item[(ii)] $\Supp B_{M'}$ is a simple normal crossing 
divisor on $M'$, and 
\item[(iii)] $\Supp g_*^{-1}(D+Y)\cup \Exc (g)$ is also a 
simple normal crossing divisor on $M'$. 
\end{itemize}
Let $Y'$ be the union of the irreducible components 
of $B_{M'}^{=1}$ that are mapped into $Y$ by $g$. 
We put 
$
(K_{M'}+B_{M'})|_{Y'}=K_{Y'}+B_{Y'}$.  
Then 
$
g_*\mathcal O_{Y'}(\lceil -(B_{Y'}^{<1})\rceil 
-\lfloor B_{Y'}^{>1}\rfloor)\simeq 
\mathcal O_{Y}(\lceil -(B_{Y}^{<1})\rceil 
-\lfloor B_{Y}^{>1}\rfloor)
$ 
by Proposition \ref{f-prop4.1}. 
This implies that 
$\mathcal I_{X_{-\infty}}\overset{\simeq}{\longrightarrow}
f_*g_*\mathcal O_{Y'}(\lceil -(B_{Y'}^{<1})\rceil 
-\lfloor B_{Y'}^{>1}\rfloor)$.  
By construction, 
$
K_{Y'}+B_{Y'}=g^*(K_Y+B_Y)\sim _{\mathbb R}g^*f^*\omega$.  
By Proposition \ref{f-prop4.2}, 
the collection of subvarieties 
$\{C\}$ in Definition \ref{def3.2} coincides with 
the image of $(Y', B_{Y'})$-strata that are not contained in $X_{-\infty}$. 
Therefore, by replacing $M$ and $(Y, B_Y)$ with 
$M'$ and $(Y', B_{Y'})$, we may assume that 
the ambient space $M$ is quasi-projective. 
\end{proof}

\begin{lem}\label{f-lem4.4}
Let $(Y, B_Y)$ be a simple normal crossing 
pair. 
Let $V$ be a smooth 
variety such that 
$Y\subset V$. 
Then we can construct a sequence of 
blow-ups 
$$V_k\to V_{k-1}\to\cdots\to V_0=V$$ with 
the following properties. 
\begin{itemize}
\item[$(1)$] $\sigma_{i+1}:V_{i+1}\to V_i$ is the 
blow-up along a smooth irreducible 
component of $\Supp B_{Y_i}$ for 
every $i\geq 0$. 
\item[$(2)$] We put $Y_0=Y$ and $B_{Y_0}=B_Y$. 
Let $Y_{i+1}$ be the 
strict transform 
of $Y_i$ for every $i\geq 0$. 
\item[$(3)$] 
We define $K_{Y_{i+1}}+B_{Y_{i+1}}=\sigma^*_{i+1}
(K_{Y_i}+B_{Y_i})$ for 
every $i\geq 0$.  
\item[$(4)$] There exists 
an $\mathbb R$-divisor $D$ on $V_k$ such 
that $D|_{Y_k}=B_{Y_k}$. 
\item[$(5)$] $\sigma_*\mathcal O_{Y_k}(\lceil 
-(B^{<1}_{Y_k})\rceil-\lfloor B^{>1}_{Y_k}\rfloor)\simeq 
\mathcal O_Y(\lceil-(B^{<1}_Y)\rceil-\lfloor B^{>1}_Y\rfloor)$, 
where $\sigma: V_k\to V_{k-1}\to\cdots\to V_0=V$.  
\end{itemize}
\end{lem}

\begin{proof} 
It is sufficient to check (5). 
All the other properties are obvious by the construction of the 
sequence of blow-ups. 
By an easy calculation of discrepancy coefficients similar to 
the proof of Proposition \ref{f-prop4.1}, 
we can check that $$\sigma_{i+1*}\mathcal O_{V_{i+1}}(\lceil -(B_{Y_{i+1}}^{<1})
\rceil 
-\lfloor B_{Y_{i+1}}^{>1}\rfloor)\simeq 
\mathcal O_{V_{i}}(\lceil -(B_{Y_{i}}^{<1})\rceil 
-\lfloor B_{Y_{i}}^{>1}\rfloor)$$ for every $i$. 
This implies the desired isomorphism. 
\end{proof}

We can easily check: 

\begin{lem}\label{f-lem4.5} 
In Lemma \ref{f-lem4.4}, 
let $C'$ be a stratum of $(Y_k, B_{Y_k})$. 
Then $\sigma(C')$ is a stratum of $(Y, B_Y)$. 
Let $C$ be a stratum of $(Y, B_Y)$. 
Then there is a stratum $C'$ of $(Y_k, B_{Y_k})$ such that 
$\sigma(C')=C$. 
\end{lem}

The following lemma is easy but 
very useful (cf.~\cite[Proposition 10.59]{kollar2}). 

\begin{lem}\label{f-lem4.6}
Let $Y$ be a simple normal crossing 
variety. Let $V$ be a smooth 
quasi-projective variety such that 
$Y\subset V$. Let $\{P_i\}$ be 
any finite set of closed points of 
$Y$. 
Then we can find a quasi-projective 
variety $W$ such that 
$Y\subset W\subset V$, $\dim W=\dim Y+1$, and 
$W$ is smooth 
at $P_i$ for every $i$. 
\end{lem}

For the proof, see, for example, Step 2 in the proof of 
\cite[Theorem 1.2]{fujino-slc}. We note that we can not always make 
$W$ smooth in Lemma \ref{f-lem4.6}. 

\begin{ex}[{\cite[Example 3.62]{fujino-book}}]\label{f-ex4.7}
Let $V\subset \mathbb P^5$ be the Segre 
embedding of $\mathbb P^1\times \mathbb P^2$. 
In this case, there are no smooth 
hypersurfaces of $\mathbb P^5$ containing $V$. 
We can check it as follows. 

If there exists a smooth hypersurface $S$ such that 
$V\subset S\subset \mathbb P^5$, then 
$\rho (V)=\rho (S)=\rho (\mathbb P^5)=1$ by 
the Lefschetz hyperplane theorem. 
It is a contradiction because $\rho(V)=2$. 
\end{ex}

By the above results, we can prove the final lemma in this section. 

\begin{lem}\label{f-lem4.8}
Let $(Y, B_Y)$ be a simple normal 
crossing 
pair such that 
$Y$ is quasi-projective. Then 
there exist a globally embedded simple normal 
crossing 
pair $(Z, B_Z)$ and 
a morphism 
$\sigma: Z\to Y$ such that 
$
K_Z+B_Z=\sigma^*(K_Y+B_Y) 
$
and 
$\sigma _*\mathcal O_Z(\lceil -(B^{<1}_Z)\rceil
-\lfloor B^{>1}_Z\rfloor)
\simeq \mathcal O_Y(\lceil -(B^{<1}_Y)\rceil
-\lfloor B^{>1}_Y\rfloor)$. 
Moreover, let $C'$ be a stratum of $(Z, B_Z)$. 
Then $\sigma(C')$ is a stratum of $(Y, B_Y)$ or $\sigma(C')$ is 
contained in $\Supp B_Y^{>1}$. 
Let $C$ be a stratum of $(Y, B_Y)$. 
Then there exists a stratum $C'$ of $(Z, B_Z)$ such 
that $\sigma(C')=C$. 
\end{lem}
\begin{proof}
Let $V$ be a smooth quasi-projective 
variety such that 
$Y\subset V$. 
By Lemma \ref{f-lem4.4} and Lemma \ref{f-lem4.5}, 
we may assume 
that there exists an $\mathbb R$-divisor 
$D$ on $V$ such that 
$D|_Y=B_Y$. 
Then we apply Lemma \ref{f-lem4.6}. 
We can find a quasi-projective 
variety $W$ such that 
$Y\subset W\subset V$, $\dim W=\dim Y+1$, and 
$W$ is smooth at the generic point of any stratum 
of $(Y, \Supp B_Y)$. 
Of course, we can make $W\not\subset \Supp D$ 
by the proof of Lemma \ref{f-lem4.6}. 
We apply 
Hironaka's resolution to $W$ and 
use Szab\'o's resolution lemma (see, for example, 
\cite[3.5 Resolution lemma]{fujino-what}). 
More precisely, we take blow-ups outside $U$, where 
$U$ is the largest Zariski open set 
of $W$ such that 
$(Y, B_Y)|_U$ is a globally embedded simple normal 
crossing pair. 
Then we obtain a desired globally embedded 
simple normal crossing 
pair $(Z, B_Z)$. 
Precisely speaking, we can check that 
$(Z, B_Z)$ has the desired properties 
by an easy calculation 
of discrepancy coefficients 
similar to the proof of Proposition \ref{f-prop4.1}. 
\end{proof}

\begin{thm}\label{f-thm4.9}
In Definition \ref{def3.2}, 
it is sufficient to assume that $(Y, B_Y)$ is 
a quasi-projective $($not necessarily embedded$)$ 
simple normal crossing pair. \end{thm}

\begin{proof}
We only assume that $(Y, B_Y)$ is a simple normal crossing pair 
in Definition \ref{def3.2}. We assume that $Y$ is quasi-projective. 
Then we apply Lemma \ref{f-lem4.8} to $(Y, B_Y)$. 
Let $\sigma:(Z, B_Z)\to (Y, B_Y)$ be as in Lemma \ref{f-lem4.8}. 
Then 
$
\bigl(X, \omega, f\circ \sigma: (Z, B_Z)\to X\bigr)
$ is a quasi-log scheme in the sense of Definition \ref{def3.2}. 
\end{proof}

Proposition \ref{f-prop4.10} shows that 
it is not so easy to apply Chow's lemma 
directly to make $(Y, B_Y)$ quasi-projective in Definition \ref{def3.2}. 

\begin{prop}[{\cite[Proposition 3.65]{fujino-book}}]\label{f-prop4.10}
There exists a complete simple normal crossing 
variety $Y$ with the following property. 
If $f:Z\to Y$ is a proper surjective morphism from 
a simple normal crossing variety $Z$ such 
that $f$ is an isomorphism 
over the generic point of any stratum of $Y$, then 
$Z$ is non-projective. 
\end{prop}

\begin{proof}
We take a smooth complete non-projective toric variety 
$X$. 
We put $V=X\times \mathbb P^1$. Then 
$V$ is a toric variety. 
We consider $Y=V\setminus T$, where 
$T$ is the big torus of $V$. 
We will see that $Y$ has the desired property. 
By the above construction, there is 
an irreducible component $Y'$ of $Y$ that is 
isomorphic to $X$. 
Let $Z'$ be the irreducible component of $Z$ mapped 
onto $Y'$ by $f$. 
So, it is sufficient to see that 
$Z'$ is not projective. 
On $Y'\simeq X$, there is a torus invariant effective 
one cycle $C$ such that 
$C$ is numerically trivial. 
By construction and the assumption, $g=f|_{Z'}: 
Z'\to Y'\simeq X$ is birational and an isomorphism 
over the generic point of any torus invariant curve 
on $Y'\simeq X$. 
We note that any torus invariant curve on $Y'\simeq 
X$ is a stratum of $Y$. 
We assume that $Z'$ is projective, 
then there is a very ample 
effective divisor $A$ on $Z'$ such that 
$A$ does not contain any irreducible components of 
the inverse image of $C$. 
Then $B=f_*A$ is an effective Cartier divisor 
on $Y'\simeq X$ such that 
$\Supp B$ contains no irreducible 
components of $C$. 
It is a contradiction because $\Supp B\cap C\ne \emptyset$ 
and $C$ is numerically trivial. 
\end{proof}

Proposition \ref{f-prop4.10} is the main 
reason why we proved Theorem \ref{f-thm4.3} 
for the proof of our main theorem:~Theorem \ref{f-thm1.1} 
and Theorem \ref{f-thm3.5}. 
Now the proof of Theorem \ref{f-thm1.1} is almost obvious. 

\begin{proof}[Proof of Theorem \ref{f-thm3.5}]
Let $f:(Y, B_Y)\to X$ be a quasi-log resolution as in Definition \ref{def3.2}. 
By Theorem \ref{f-thm4.3}, 
we may assume that $Y$ is quasi-projective. 
We consider the fiber product $Y'=Y\times _XX'$.  
$$
\xymatrix{
Y' \ar[r]^{h'}\ar[d]_{f'}&Y\ar[d]^{f}\\
X' \ar[r]_{h}&X
} 
$$
We put $B_{Y'}=h'^*B_Y$. 
Then $(Y', B_{Y'})$ is a quasi-projective simple normal crossing pair because 
$h$ is a smooth quasi-projective 
morphism and $(Y, B_Y)$ is a quasi-projective simple normal crossing 
pair. 
Since $K_Y+B_Y\sim _{\mathbb R}f^*\omega$, 
we have 
$$
f'^*\omega'=f'h^*\omega\otimes f'^*\omega_{X'/X}
=h'f^*\omega\otimes \omega_{Y'/Y}
\sim_{\mathbb R}h'^*(K_Y+B_Y)\otimes \omega_{Y'/Y}
=K_{Y'}+B_{Y'}. 
$$ Note that $\omega_{X'/X}$ is trivial when $h$ is \'etale. 
By the flat base change theorem, 
we have 
\begin{align*}
h^*\mathcal I_{X_{-\infty}}&=h^*f_*
\mathcal O_Y(\lceil -(B_Y^{<1})\rceil 
-\lfloor B_Y^{>1}\rfloor)\simeq f'_*h'^*
\mathcal O_Y(\lceil -(B_Y^{<1})\rceil -\lfloor B_Y^{>1}\rfloor)\\
&\simeq f'_*\mathcal O_{Y'}(\lceil -(B_{Y'}^{<1})\rceil -\lfloor B_{Y'}^{>1}\rfloor). 
\end{align*}
Finally, by Theorem \ref{f-thm4.9}, we may assume that 
$(Y', B_{Y'})$ is a globally embedded simple normal crossing 
pair. Therefore, 
$
\bigl(X', \omega', f':(Y', B_{Y'})\to X'\bigr) 
$ 
gives us the desired quasi-log structure. 
\end{proof}

\section{An application to quasi-log canonical Fano varieties}\label{f-sec5}

Let us recall the vanishing theorem for projective qlc pairs. 

\begin{thm}[Vanishing theorem for qlc pairs]\label{f-thm5.1}
Let $[X, \omega]$ be a projective 
qlc pair and let $L$ be a Cartier divisor on $X$ such that 
$L-\omega$ is ample. 
Then $H^i(X, \mathcal O_X(L))=0$ for every $i>0$. 
\end{thm}

We give a proof of Theorem \ref{f-thm5.1} for the reader's convenience. 

\begin{proof}
Let $f:(Y, B_Y)\to X$ be a quasi-log resolution as in Definition \ref{def3.2}. 
Since $[X, \omega]$ is qlc, $B_Y=B^{\leq 1}_Y$ holds. 
Then 
$$
f^*(L-\omega)\sim _{\mathbb R}f^*L-(K_Y+B_Y)
=f^*L+\lceil -(B_Y^{<1})\rceil-(K_Y+\{B_Y\}+B_Y^{=1}) 
$$
because $B_Y=B_Y^{\leq 1}$. Therefore, we have 
$
H^i(X, f_*\mathcal O_Y(f^*L+\lceil -(B_Y^{<1})\rceil))=0
$ 
for every $i>0$ by \cite[Theorem 1.1 (ii)]{fujino-vanishing}. 
Note that 
$
f_*\mathcal O_Y(f^*L+\lceil -(B_Y^{<1})\rceil)
\simeq \mathcal O_X(L)\otimes f_*\mathcal O_Y(\lceil 
-(B_Y^{<1})\rceil)\simeq \mathcal O_X(L)$ 
because $X_{-\infty}=\emptyset$. 
This implies that 
$
H^i(X, \mathcal O_X(L))=0
$ 
for every $i>0$. 
\end{proof}

By combining Theorem \ref{f-thm5.1} with 
Theorem \ref{f-thm3.5}, 
we can easily check Corollary \ref{f-cor1.2}. 

\begin{proof}[{Proof of Corollary \ref{f-cor1.2}}]
Without loss of generality, we may assume that 
$X$ is connected. 
Since $-\omega$ is ample, $H^i(X, \mathcal O_X)=0$ for 
every $i>0$ by Theorem \ref{f-thm5.1}. 
Therefore, we have $\chi(X, \mathcal O_X)=1$. 
Assume there exists a nontrivial finite \'etale morphism $f: 
\widetilde X\to X$ from a connected scheme $\widetilde X$. 
By Theorem \ref{f-thm3.5}, the pair $[\widetilde X, \widetilde \omega]$, 
where $\widetilde \omega=f^*\omega$, 
is a qlc pair such that $-\widetilde \omega$ is ample. 
Thus, $H^i(\widetilde X, \mathcal O_{\widetilde X})=0$ for 
every $i>0$ by Theorem \ref{f-thm5.1} again. 
This implies $\chi (\widetilde X, \mathcal O_{\widetilde X})=1$. 
By the Riemann--Roch formula (\cite[Example 18.3.9]{fulton}), 
we have 
$
\chi (\widetilde X, \mathcal O_{\widetilde X})=\deg f\cdot \chi (X, \mathcal O_X)$.  
Therefore, we obtain $\deg f=1$, a contradiction. 
This means that $X$ has no nontrivial 
finite 
\'etale covers, or equivalently, the algebraic fundamental 
group of $X$ is trivial. 
\end{proof}

As a direct consequence of Corollary \ref{f-cor1.2} and 
the main theorem of \cite{fujino-slc}, we have: 

\begin{cor}\label{f-cor5.2} 
Let $(X, \Delta)$ be a projective semi-log canonical 
pair such that $-(K_X+\Delta)$ is ample. 
Then the algebraic fundamental group of $X$ is trivial. 
\end{cor}

\begin{proof}
By \cite{fujino-slc}, $[X, K_X+\Delta]$ has a natural quasi-log structure 
with only qlc singularities. 
Therefore, Corollary \ref{f-cor5.2} is a special case of Corollary \ref{f-cor1.2}. 
\end{proof}

Note that a union of some slc strata of a log Fano pair with only semi-log 
canonical singularities 
is a quasi-log canonical Fano variety by Example \ref{f-ex5.3}. 

\begin{ex}\label{f-ex5.3}
Let $(X, \Delta)$ be a connected 
projective semi-log canonical pair such that 
$-(K_X+\Delta)$ is ample. 
Let $W$ be a union of some slc strata of $(X, \Delta)$ with 
the reduced scheme structure. Then 
$[W, \omega]$, where $\omega=(K_X+\Delta)|_W$, is a projective 
qlc pair such that $-\omega$ is ample by adjunction 
(\cite[Theorem 1.13]{fujino-slc}). 
By \cite[Theorem 1.11]{fujino-slc}, we obtain $H^1(X, \mathcal I_W)=0$ 
where $\mathcal I_W$ is the defining ideal sheaf of $W$ on $X$. 
Therefore, we obtain $H^0(W, \mathcal O_W)=\mathbb C$ 
by the surjection $\mathbb C=H^0(X, \mathcal O_X)\to H^0(W, \mathcal O_W)$. 
This implies that $W$ is connected. 
\end{ex}

The author learned the following example from Tetsushi Ito. 

\begin{ex}[Topological versus algebraic]\label{f-ex5.4}
We consider the Higman group $G$. 
It is generated by $4$ elements $a$, $b$, $c$, $d$ with 
the relations
\begin{align*}
a^{-1}ba=b^2, \quad b^{-1}cb=c^2, \quad c^{-1}dc=d^2, \quad d^{-1}ad=a^2. 
\end{align*}
It is well known that $G$ has no nontrivial 
finite quotients. 
By \cite[Theorem 12.1]{simpson}, 
there is an irreducible projective variety $X$ such that 
$\pi_1(X)\simeq G$. 
In this case, the algebraic fundamental group of $X$, which is 
the profinite completion of $\pi_1(X)$, is trivial. 
\end{ex}

Example \ref{f-ex5.4} shows that 
Conjecture \ref{f-conj1.3} does not directly follow from Corollary \ref{f-cor1.2}. 
We give a nontrivial example of reducible log Fano pairs with only semi-log 
canonical singularities. 

\begin{ex}\label{f-ex5.5}
We consider the lattice $N=\mathbb Z^3$. 
Let $n$ be an integer with $n\geq 3$. 
We consider a convex polyhedron $P$ in $N_{\mathbb R}=N\otimes 
\mathbb R\simeq \mathbb R^3$ whose 
vertices are $v_0, v_1, \ldots, v_n\in N$ such that 
$v_0=(0, 0, -1)$ and that the third coordinates 
of $v_1, \ldots, v_n$ are $1$. 
Assume that 
$P$ contains $(0, 0, 0)$ in its interior. 
Then the cones spanned by $(0, 0, 0)$ and faces of $P$ subdivide 
$\mathbb R^3$ into 
$n+1$ three-dimensional cones. 
This subdivision of $\mathbb R^3$ corresponds to a complete toric 
threefold $X$. 
Then we have the following properties. 
\begin{itemize}
\item[(1)] $-K_X$ is ample since $P$ is convex. 
\item[(2)] $D_0\sim D_1+\cdots +D_n$ and $D_0$ is $\mathbb Q$-Cartier, 
where $D_i$ is the torus invariant prime divisor 
on $X$ associated to $v_i$ for every $i$. 
\item[(3)] Let $x\in X$ be the torus invariant 
closed point associated to the cone spanned by $v_1, v_2, \ldots, v_n$. 
Then $X\setminus x$ is $\mathbb Q$-factorial, but 
$X$ is not $\mathbb Q$-factorial when $n\geq 4$. 
\item[(4)] We put $\Delta=D_1+\cdots +D_n$. Then 
$(X, \Delta)$ is a log canonical 
Fano threefold. 
Note that $-(K_X+\Delta)\sim D_0$. 
\item[(5)] We put $W=\lfloor \Delta\rfloor=\Delta$ and 
$
K_W+\Delta_W=(K_X+\Delta)|_W$.  
Then $(W, \Delta_W)$ is a two-dimensional 
log Fano pair with only semi-log canonical singularities. 
Note that $W$ is Cohen--Macaulay since $W$ is 
$\mathbb Q$-Cartier. 
\end{itemize} 
This $W$ shows that the number of 
irreducible components of log Fano pairs with only semi-log 
canonical singularities is not 
bounded. 
\end{ex}

We recommend the reader who can read Japanese 
to see \cite{fujino-fano} for some related topics and open problems 
on singular Fano varieties. 

\section{Simple connectedness of log canonical log Fano 
pairs}\label{f-sec6}

In this section, we prove that a 
log Fano pair with only log canonical singularities 
is always simply connected. 
Theorem \ref{f-thm6.1} is Fujita's answer to the author's question. 

\begin{thm}[Fujita]\label{f-thm6.1} 
Let $(X, \Delta)$ be a projective log canonical 
pair such that $-(K_X+\Delta)$ is ample. 
Then $X$ is simply connected.  
\end{thm}
\begin{proof} 
First of all, we may assume that $X$ is connected. 
Without loss of generality, we may assume that 
$\Delta$ is a $\mathbb Q$-divisor by perturbing $\Delta$ slightly. 
Then, by \cite[Corollary 1.3 (2)]{hacon-mckernan}, 
$X$ is rationally chain connected. 
Since $X$ is normal and rationally chain connected, $\pi_1(X)$ is finite 
(\cite[4.13 Theorem]{kollar}). 
Let $f: \widetilde X\to X$ be the universal cover of $X$. 
Since $\pi_1(X)$ is finite, $f$ is finite and \'etale. 
It is obvious that 
$(\widetilde X, \widetilde \Delta)$ is log canonical and 
$-(K_{\widetilde X}+\widetilde \Delta)$ is ample, where 
$
K_{\widetilde X}+\widetilde \Delta=f^*(K_X+\Delta)$. 
By \cite[Theorem 8.1]{fujino}, we have 
$
H^i(X, \mathcal O_X)=H^i(\widetilde X, \mathcal O_{\widetilde X})=0
$ 
for every $i>0$. This implies 
$
\chi(X, \mathcal O_X)=\chi (\widetilde X, \mathcal O_{\widetilde X})=1$.  
On the other hand, 
$
\chi(\widetilde X, \mathcal O_{\widetilde X})=\deg f\cdot \chi(X, \mathcal O_X) 
$ 
holds by the Riemann--Roch formula 
(\cite[Example 18.3.9]{fulton}). 
Thus we obtain $\deg f=1$. 
Therefore, $X$ is simply connected. 
\end{proof}

\begin{rem}\label{f-rem6.2}
By \cite[Corollary 1.3 (2)]{hacon-mckernan}, we can 
easily see that a 
log Fano pair with only semi-log canonical singularities 
is rationally chain connected. 
However, \cite[4.13 Theorem]{kollar} does not always 
hold for 
{\em{non-normal}} rationally chain connected varieties. 
Note that a nodal rational curve $C$ is rationally chain connected 
such that $\pi_1(C)$ is infinite. 
Therefore, the proof of Theorem \ref{f-thm6.1} does not 
work for log Fano pairs with only semi-log canonical singularities. 
\end{rem}

The following well-known example shows some subtleties 
on log Fano pairs with only log canonical singularities. 
Example \ref{f-ex6.3} says that Theorem \ref{f-thm6.1} 
does not always hold when $-(K_X+\Delta)$ is only nef and 
big. 

\begin{ex}\label{f-ex6.3}
Let $C\subset \mathbb P^2$ be a smooth 
cubic curve and let $X\subset \mathbb P^3$ be the cone over 
$C\subset \mathbb P^2$. 
Then $X$ is a Gorenstein log canonical surface such that 
$-K_X$ is ample. 
It is easy to see that $X$ is rationally chain connected and 
that $\pi_1(X)=\{1\}$ by Theorem \ref{f-thm6.1}. 
Let $f:Y\to X$ be the blow-up at $P$ where $P$ is the vertex of $X$. 
Then $
K_Y+E=f^*K_X$.  
The pair $(Y, E)$ is purely log terminal and $-(K_Y+E)$ is big and semiample. 
Note that the exceptional curve $E$ is isomorphic to $C$ and that $Y$ is a $\mathbb P^1$-bundle 
over $C$. 
Therefore, it is easy to see that $Y$ is not rationally chain connected and 
$\pi_1(Y)\ne \{1\}$. 
\end{ex}

Example \ref{f-ex6.4} is a nontrivial example 
of {\em{irreducible non-normal}} semi-log canonical Fano varieties. 

\begin{ex}\label{f-ex6.4}
We put $X=(x^2w-zy^2=0)\subset \mathbb P^3$. 
Then $X$ is a Gorenstein Fano variety with only 
semi-log canonical singularities. 
Note that $X$ is irreducible and non-normal. 
By using the van Kampen theorem, 
we see that $\pi_1(X)=\{1\}$. 
\end{ex}

\section{Appendix:~Ambro's original definition}\label{f-sec7} 

In this section, we prove that 
our definition of quasi-log schemes (Definition \ref{def3.2}) 
is equivalent to Ambro's original definition in \cite{ambro}. 

First, let us recall the definition of {\em{normal crossing pairs}}. 
We need it for Ambro's original definition of quasi-log schemes in \cite{ambro}. 

\begin{defn}[Normal crossing pairs]\label{f-def7.1}
A variety $X$ has {\em{normal crossing}} singularities 
if, for every closed point $x\in X$, 
$$
\widehat{\mathcal O}_{X,x}\simeq \frac{\mathbb C[[x_0, \cdots, x_{N}]]}
{(x_0\cdots x_k)}
$$
for some $0\leq k \leq N$, where 
$N=\dim X$. 
Let $X$ be a normal crossing variety. 
We say that a reduced divisor $D$ on $X$ 
is {\em{normal crossing}} if, 
in the above notation, 
we have 
$$
\widehat{\mathcal O}_{D, x}\simeq \frac{\mathbb C[[x_0, \cdots, 
x_{N}]]}{(x_0\cdots x_k, x_{i_1}\cdots x_{i_l})}
$$ 
for some $\{i_1, \cdots, i_l\}\subset\{k+1, \cdots, N\}$. 
A {\em{stratum}} of 
$X$ is an irreducible component of $X$ or the $\nu$-image 
of a log canonical center of $(X^\nu, \Xi)$, 
where $\nu:X^\nu\to X$ is the normalization and $K_{X^\nu}+\Xi=
\nu^*K_X$. 
A {\em{permissible}} Cartier divisor on $X$ is a Cartier divisor on $X$ whose 
support contains no strata of $X$.
A {\em{permissible}} $\mathbb R$-Cartier divisor is a finite 
$\mathbb R$-linear 
combination of permissible Cartier divisors on $X$. 
We say that the pair $(X, B)$ is a 
{\em{normal crossing pair}} if the 
following 
conditions are satisfied. 
\begin{itemize}
\item[(1)] $X$ is a normal crossing 
variety, and 
\item[(2)] $B$ is a permissible $\mathbb R$-Cartier divisor 
whose support is normal crossing on $X$. 
\end{itemize}

We say that a normal crossing 
pair $(X, B)$ is {\em{embedded}} 
if there exists a closed embedding 
$\iota:X\to M$, where 
$M$ is a smooth variety 
of dimension $\dim X+1$. 
We call $M$ the {\em{ambient space}} of $(X, B)$.  
We put 
$$K_{X^\nu}+\Theta=\nu^*(K_X+B), $$ 
where 
$\nu:X^\nu\to X$ is the normalization of $X$. 
A {\em{stratum}} of $(X, B)$ 
is an irreducible 
component of $X$ or 
the $\nu$-image of some log canonical 
center of $(X^\nu, \Theta)$ on $X$. 
\end{defn}

It is obvious that a {\em{simple normal crossing 
pair}} in Definition \ref{f-def2.4} is a {\em{normal crossing 
pair}} in Definition \ref{f-def7.1}. 
Note that the differences between normal crossing varieties and 
simple normal crossing varieties sometimes 
cause some subtle troubles. 
For the details, see, for example, \cite[3.6 Whitney umbrella]{fujino-what}. 

Let us recall Ambro's original definition of 
quasi-log schemes in \cite{ambro}. 

\begin{defn}[Quasi-log schemes]\label{f-def7.2}
A {\em{quasi-log scheme}} is a scheme $X$ endowed with an 
$\mathbb R$-Cartier divisor 
(or $\mathbb R$-line bundle) $\omega$, a proper closed subscheme 
$X_{-\infty}\subset X$, and a finite collection $\{C\}$ of reduced 
and irreducible subschemes of $X$ such that there is a 
proper morphism $f:(Y, B_Y)\to X$ from an {\em{embedded 
normal crossing pair}} 
satisfying the following properties: 
\begin{itemize}
\item[(1)] $f^*\omega\sim_{\mathbb R}K_Y+B_Y$. 
\item[(2)] The natural map 
$\mathcal O_X
\to f_*\mathcal O_Y(\lceil -(B_Y^{<1})\rceil)$ 
induces an isomorphism 
$$
\mathcal I_{X_{-\infty}}\overset{\simeq}{\longrightarrow} f_*\mathcal O_Y(\lceil 
-(B_Y^{<1})\rceil-\lfloor B_Y^{>1}\rfloor),  
$$ 
where $\mathcal I_{X_{-\infty}}$ is the defining ideal sheaf of 
$X_{-\infty}$. 
\item[(3)] The collection of subvarieties $\{C\}$ coincides with the image 
of $(Y, B_Y)$-strata that are not included in $X_{-\infty}$. 
\end{itemize}
\end{defn}

In Definition \ref{def3.2}, we assume that 
$(Y, B_Y)$ is a {\em{globally embedded simple normal crossing pair}}. 
On the other hand, in Definition \ref{f-def7.2}, we only assume that 
$(Y, B_Y)$ is an {\em{embedded normal crossing pair}}. 

\begin{rem}[Schemes versus varieties]\label{f-rem7.3}
A quasi-log {\em{scheme}} is called 
a quasi-log {\em{variety}} in \cite{ambro}. 
However, $X$ is not always reduced when $X_{-\infty}\ne \emptyset$. 
Note that 
$X$ is reduced when $X_{-\infty}=\emptyset$. 
\end{rem}

\begin{ex}[{\cite[Examples 4.3.4]{ambro}}]\label{f-ex7.4}
Let $X$ be an effective Cartier divisor 
on a smooth variety $M$ such that $\Supp X$ is a simple normal crossing 
divisor. 
Assume that $Y$, the reduced part of $X$, is non-empty. 
We put $\omega=(K_M+X)|_X$. 
Let $X_{-\infty}$ be the union of the non-reduced components of $X$. 
We put $K_Y+B_Y=(K_M+X)|_Y$. 
Let $f:Y\to X$ be the closed embedding. 
Then $
\bigl( X, \omega, f:(Y, B_Y)\to X\bigr)
$ 
is a quasi-log scheme. 
Note that $X$ has 
non-reduced irreducible components if $X_{-\infty}\ne \emptyset$. 
We also note that $f$ is not surjective 
if $X_{-\infty}\ne \emptyset$. 
\end{ex}

Lemma \ref{f-lem7.6} is essentially the same as Ambro's 
{\em{embedded log transformations}} in \cite{ambro}. 

\begin{lem}\label{f-lem7.6} 
Let $(Y, B_Y)$ be an embedded normal crossing 
pair and let $M$ be the ambient space of $(Y, B_Y)$. 
Then there are a projective surjective morphism 
$\sigma:M'\to M$ from a smooth 
variety $M'$ such that 
$\sigma$ is a composition of blow-ups and 
a simple normal crossing 
pair $(Z, B_Z)$ embedded into $M'$ with the following 
properties. 
\begin{itemize}
\item[(i)] $\sigma:Z\to Y$ is surjective 
and $K_Z+B_Z=\sigma^*(K_Y+B_Y)$. 
\item[(ii)] $\sigma_*\mathcal O_Z(\lceil -(B_Z^{<1})\rceil-\lfloor B_Z^{>1}\rfloor)
\simeq \mathcal O_Y(\lceil -(B_Y^{<1})\rceil -\lfloor B_Y^{>1}\rfloor)$. 
\item[(iii)] Let $C'$ be a stratum of $(Z, B_Z)$. 
Then $\sigma(C')$ is a stratum of $(Y, B_Y)$ or 
is contained in $\Supp B_Y^{>1}$. 
Let $C$ be a stratum of $(Y, B_Y)$. 
Then there is a stratum $C'$ of $(Z, B_Z)$ such that 
$\sigma(C')=C$. 
\end{itemize}
\end{lem}
\begin{proof}
First, we can construct a sequence of blow-ups 
$M_k\to M_{k-1}\to\cdots\to M_0=M$ with 
the following properties. 
\begin{itemize}
\item[(a)] $\sigma_{i+1}:M_{i+1}\to M_i$ is the blow-up 
along a smooth stratum of $Y_i$ for every $i$. 
\item[(b)] We put $Y_0=Y$, $B_{Y_0}=B_Y$, and 
$Y_{i+1}=\sigma_{i+1}^{-1}(Y_i)$ with the reduced scheme structure. 
\item[(c)] $Y_k$ is a simple normal crossing divisor on $M_k$. 
\end{itemize}
We can check that $K_{Y_{i+1}}=\sigma_{i+1}^*K_{Y_i}$ for 
every $i$ by construction. We can directly check that 
$R^1\sigma_{i+1*}\mathcal O_{M_{i+1}}(-Y_{i+1})=0$ and 
$\sigma_{i+1*}\mathcal O_{M_{i+1}}(-Y_{i+1})\simeq 
\mathcal O_{M_i}(-Y_i)$ for every $i$. 
Therefore, by the diagram: 
$$
\xymatrix{
0\ar[r]&\mathcal O_{M_i}(-Y_i)\ar[r]\ar[d]^{\simeq}
&\mathcal O_{M_i}\ar[r]\ar[d]^{\simeq}&\mathcal O_{Y_i} \ar[d]\ar[r]&0\ \\
0\ar[r]&\sigma_{i+1*}\mathcal O_{M_{i+1}}(-Y_{i+1})
\ar[r]&\sigma_{i+1*}\mathcal O_{M_{i+1}}
\ar[r]&\sigma_{i+1*}\mathcal O_{Y_{i+1}}\ar[r]&0, 
}
$$ 
we obtain $\sigma_{i+1*}\mathcal O_{Y_{i+1}}\simeq \mathcal O_{Y_i}$ for every $i$. 
We put $B_{Y_{i+1}}=\sigma_{i+1}^*B_{Y_i}$ for every $i$. 
Then, by replacing $(Y, B_Y)$ and $M$ with 
$(Y_k, B_{Y_k})$ and $M_k$, we may assume that 
$Y$ is a simple normal crossing divisor on $M$. 

Next, we can construct a sequence of blow-ups 
$M_k\to M_{k-1}\to\cdots\to M_0=M$ with 
the following properties. 
\begin{itemize}
\item[(1)] $\sigma_{i+1}:M_{i+1}\to M_i$ is the blow-up 
along a smooth stratum of $(Y_i, \Supp B_{Y_i})$ contained 
in $\Supp B_{Y_i}$ for every $i$. 
\item[(2)] We put $Y_0=Y$ and $B_{Y_0}=B_Y$. 
Let $Y_{i+1}$ be the strict transform of $Y_i$ on $M_{i+1}$ for every $i$. 
\item[(3)] We put $K_{Y_{i+1}}+B_{Y_{i+1}}=\sigma_{i+1}^*(K_{Y_i}+B_{Y_i})$ for 
every $i$. 
\item[(4)] $\Supp B_{Y_k}$ is a simple normal crossing divisor on $Y_k$. 
\end{itemize} 

Finally, by construction, we can check the properties (i), (ii), and (iii) 
for $\sigma:M_k\to M$ and $(Y_k, B_{Y_k})$ by an easy calculation of 
discrepancy coefficients similar to the proof of Proposition \ref{f-prop4.1}. 
\end{proof}

\begin{prop}\label{f-prop7.7} 
Assume that 
$(Y, B_Y)$ is an embedded simple normal crossing 
pair in Definition \ref{f-def7.2}. 
Let $M$ be the ambient space of $(Y, B_Y)$. 
Then, by taking some sequence of blow-ups of $M$, 
we may further assume that 
$(Y, B_Y)$ is a globally embedded simple normal crossing 
pair in Definition \ref{f-def7.2}. 
\end{prop}
\begin{proof} 
It is sufficient to apply Lemma \ref{f-lem4.4} and Lemma \ref{f-lem4.5} by 
putting $V=M$. If $B_Y=B^{\leq 1}_Y$, 
then this proposition is nothing but \cite[Lemma 3.3]{fujino-slc}. 
\end{proof}

Therefore, by Lemma \ref{f-lem7.6} and Proposition \ref{f-prop7.7}, 
Definition \ref{def3.2} is equivalent to Ambro's original 
definition of quasi-log schemes:~Definition \ref{f-def7.2}. 

\begin{thm}\label{f-thm7.8}
Definition \ref{def3.2} is equivalent to Definition \ref{f-def7.2}. 
\end{thm} 

\end{document}